\documentclass[11pt]{amsart}
\usepackage{amsthm}
\usepackage{amssymb}
\usepackage{amsmath}
\usepackage{tikz}
\usepackage{url}
\usepackage{cite}
\usepackage{a4wide}
\usepackage{subcaption}
\theoremstyle{plain}
\newtheorem{thm}{{Theorem}}[section]

\newtheorem{cor}[thm]{Corollary}

\newtheorem{prop}[thm]{Proposition}

\newtheorem{rem}[thm]{Remark}

\newcommand{\e}{\boldsymbol{e}}
\newcommand{\Oh}{\mathcal{O}}
\newcommand{\Ex}{\mathbb{E}}
\newcommand{\Var}{\mathbb{V}}
\newcommand{\Tree}{\mathcal{T}}
\newcommand{\Rec}{\mathcal{R}}
\newcommand{\Bin}{\mathcal{B}}
\newcommand{\rec}{\mathsf{rec}}
\newcommand{\bin}{\mathsf{bin}}
\newcommand{\BT}{\mathfrak{B}}
\newcommand{\RT}{\mathfrak{R}}

\begin{document}

\thispagestyle{plain}

\title{On the distribution of eigenvalues of increasing trees}

\author{Kenneth Dadedzi}
\address{Kenneth Dadedzi\\
Department of Mathematical Sciences\\
University of Ghana\\
Post Office Box LG 62\\ 
Legon, Accra\\
Ghana\\
and 
Department of Mathematical Sciences\\
Stellenbosch University\\
Private Bag X1\\
Matieland 7602\\
South Africa}
\email{kdadedzi@ug.edu.gh}

\thanks{The first author was supported by the German Academic Exchange Service DAAD}

\author{Stephan Wagner}
\address{Stephan Wagner\\
Department of Mathematics\\
Uppsala University\\
Box 480
751 06 Uppsala\\
Sweden\\
and
Department of Mathematical Sciences\\
Stellenbosch University\\
Private Bag X1\\
Matieland 7602\\
South Africa}
\email{stephan.wagner@math.uu.se}

\thanks{The second author was supported by the National Research Foundation of South Africa, grant 96236, and the Knut and Alice Wallenberg Foundation.}


\maketitle

\begin{abstract}
We prove that the multiplicity of a fixed eigenvalue $\alpha$ in a random recursive tree on $n$ vertices satisfies a central limit theorem with mean and variance asymptotically equal to $\mu_{\alpha} n$ and $\sigma^2_{\alpha} n$ respectively. It is also shown that $\mu_{\alpha}$ and $\sigma^2_{\alpha}$ are positive for every totally real algebraic integer. The proofs are based on a general result on additive tree functionals due to Holmgren and Janson. In the case of the eigenvalue $0$, the constants $\mu_0$ and $\sigma^2_0$ can be determined explicitly by means of generating functions.
Analogous results are also obtained for Laplacian eigenvalues and binary increasing trees. 
\end{abstract}

\section{Introduction}

The spectra of matrices associated with graphs are known to carry substantial information about the graphs. The study of graph eigenvalues has found many applications in mathematics, physics, chemistry, biology, economics and computer science (see for example \cite[Chapter 9]{DPSGraphSpectra}).

\medskip

Eigenvalues of random graphs under different models of randomness have been studied extensively due to their connections to random matrix theory, see for example \cite{Anderson2010introduction,Bai2009Spectral,Tao2012Topics}.
Wigner's semicircle law \cite{Wigner1958distribution} famously appears in the limit of the eigenvalue distribution of dense random graphs.

\medskip

The eigenvalues of trees are quite different from those of arbitrary graphs. A famous early result due to Schwenk \cite{Schwenk1973almost} states that the probability that a (uniformly) random tree has a cospectral mate (a non-isomorphic tree with the same eigenvalues) tends to $1$ as the number of vertices goes to infinity. The actual distribution of eigenvalues in random trees was more recently studied by Bhamidi, Evans and Sen in \cite{bhamidi2012spectra}. In particular, they prove (see \cite[Theorem 4.1]{bhamidi2012spectra}) that the adjacency spectrum converges to a well-defined limit under different random tree models, including e.g.~conditioned Galton--Watson trees and recursive trees. However, they do not fully characterise this limit. Recently, Salez \cite{Salez2020Spectral} studied the spectral distribution of \emph{unimodular Galton--Watson trees}, which occur prominently as weak limits of large random graphs with given asymptotic degree distribution.

\medskip

The focus of this paper is the specific model of random recursive trees. A random recursive tree with $n$ vertices is constructed by a simple growth process: start with a root labelled $1$. In the $k$-th step, a vertex labelled $k$ is attached to one of the previous $k-1$ vertices, chosen uniformly at random. There are $(n-1)!$ different labelled trees that can be obtained in this way; each of them has the same probability $\frac{1}{(n-1)!}$ in our model.

\medskip

We provide a refinement of the result of Bhamidi, Evans and Sen for individual eigenvalues by studying the multiplicity of a fixed eigenvalue $\alpha$ in the adjacency spectrum of a large random recursive tree. In particular, we prove that this multiplicity has linear mean and variance in the size of the tree, and also show that a central limit theorem holds. Analogous results are also obtained for the Laplacian spectrum as well as binary increasing trees, which are similar to recursive trees and also belong to the general family of increasing trees that was introduced by Bergeron, Flajolet and Salvy \cite{bergeron1992varieties}. The proof of these results is based on tools from linear algebra combined with a general result due to Holmgren and Janson \cite{HolmgrenJanson} on additive tree functionals (for a precise definition, see Section~\ref{sec:additive}). In a forthcoming paper \cite{Dadedzi2022spectrum}, we will use rather different methods to establish a similar central limit theorem for simply generated trees.

\medskip

In the special case of the eigenvalue $0$, which is also related to other parameters of a tree (namely the independence number and the matching number), we can provide a more detailed analysis by means of generating functions and determine the mean and variance more precisely: specifically, the mean is $ 0.192694 n+ \Oh(1)$, and the variance is $0.138629 n + \Oh(n/\log n)$ for recursive trees with $n$ vertices. For binary increasing trees, mean and variance are $0.085753 n+ \Oh(1)$ and $0.057162 n+ \Oh(1)$, respectively. See Section~\ref{sec:EV0} for details.

\section{Preliminaries}

Before we get to the main theorems, we first provide some important definitions and gather auxiliary results on tree spectra, increasing trees, and additive tree functionals in this section.

\subsection{Eigenvalues of graphs and trees}

Let us briefly review some simple but relevant facts on the spectra of graphs and trees. The eigenvalues of a graph are the eigenvalues of its adjacency matrix, the Laplacian eigenvalues the eigenvalues of its Laplacian matrix (which has the vertex degrees as diagonal entries, and otherwise $-1$ or $0$ depending on whether or not vertices are adjacent). Since both the adjacency matrix and the Laplacian matrix are symmetric, all eigenvalues are real. Moreover, the Laplacian eigenvalues are always non-negative.

\medskip

The eigenvalues of bipartite graphs, and thus in particular trees, are known to be symmetric (see for instance \cite[Theorem 3.11]{CDSSpectra}). In other words, $\alpha$ is an eigenvalue if and only if $-\alpha$ is, and the multiplicities agree.

\medskip

An important result that we will make use of in our arguments is Cauchy's interlacing theorem \cite[Corollary 1.3.12]{DPSGraphSpectra}: if $M$ is a symmetric matrix with eigenvalues $\lambda_1 \geq \lambda_2 \geq \cdots \geq \lambda_n$, and $\beta_1 \geq \beta_2 \geq \cdots \geq \beta_{n-1}$ are the eigenvalues of the matrix obtained by removing one of its rows and the corresponding column, then we have
\[\lambda_1 \geq \beta_1 \geq \lambda_2 \geq \beta_2 \geq \cdots \geq \lambda_{n-1} \geq \beta_{n-1} \geq \lambda_n.\]
This is specifically relevant for us in the case where $M$ is the adjacency matrix of a graph. Removing a row and the corresponding column is equivalent to removing one of the vertices.

\subsection{Increasing trees}\label{sec:increasing}

Increasing trees are labelled rooted trees with the characteristic property that the labels of nodes increase as one moves along any path from the root to a leaf. Bergeron, Flajolet and Salvy \cite{bergeron1992varieties} were the first to study general varieties of increasing trees. Such varieties are characterised by weights associated with increasing trees.

\medskip

Let a sequence $(w_k)_{k\geq 0}$ of non-negative real numbers with $w_0 > 0$ (called the \emph{weight sequence}) be given. We define the weight $W(T)$ of an increasing tree $T$ by
\[W(T)= \prod_{v\in V(T)} w_{d^{\star}(v)} = \prod _{k\geq 0}w_k^{D_k(T)},\] where $d^{\star}(v)$ denotes the outdegree of $v$ and $D_k(T)$ is accordingly the number of vertices in $T$ with outdegree $k$.
We associate a generating series with the weight sequence by $\Phi (t) = \sum_{k\geq 0}w_kt^k$.

\medskip

Let $\Tree_n$ be the set of all increasing rooted ordered trees of order $n$. We define the exponential generating function for increasing trees associated with the weight sequence $(w_k)_{k \geq 0}$ by
\[G(x) =\sum_{T\in \Tree_n} \frac{W(T)x^{|T|}}{|T|!}.\] 
It is known \cite[Theorem 1]{bergeron1992varieties} that this generating function satisfies the differential equation \[G^\prime (x) = \Phi (G(x)), \quad G(0)=0.\]
Once the weight has been defined, \emph{random trees} from a specific variety are obtained by selecting a tree $T$ with a probability that is proportional to its weight $W(T)$.

\medskip

For certain special varieties of increasing trees, random trees can equivalently be constructed by means of a tree evolution (growth) process. These have been characterised by Panholzer and Prodinger in \cite{panholzer2007level}: they belong to three particular groups characterised by the weight generating series $\Phi(t)$.

\begin{align*}
&\text{(1) }\Phi (t) = w_0e^{\frac{w_1}{w_0}t}, \quad \text{where  } w_0>0 \text{ and  }w_1>0.\\
&\text{(2) }\Phi (t) = w_0\left(1-\frac{w_1}{rw_0}t \right)^{-r}, \quad \text{where  } r>0,w_0>0\text{ and  }w_1>0.\\
&\text{(3) }\Phi (t) = w_0 \left( 1+ \left(\frac{w_1}{dw_0} \right)t \right)^d, \quad \text{where  } d>1,w_0>0\text{ and  }w_1>0.
\end{align*}

These correspond to the varieties of \emph{recursive trees}, \emph{generalised plane-oriented recursive trees} and \emph{$d$-ary increasing trees}, respectively. In this paper, we will focus on two particularly important varieties, namely recursive trees and binary ($2$-ary) increasing trees. It is well known that the probabilistic model of random binary increasing trees is also equivalent to that of random \emph{binary search trees}, see for instance \cite[Chapter 6]{Drmota2009random}. One usually uses the following normalised versions of the weights above (for the probabilistic model, the choice of $w_0$ and $w_1$ is actually irrelevant):
\begin{enumerate}
\item Recursive trees are associated with $\Phi (t)=e^t$ ($w_0=w_1=1$).
\item Binary increasing trees are associated with $\Phi (t)= (1+t)^2$ ($d=2$, $w_0=1$ and $w_1=2$).
\end{enumerate}

The growth process for recursive trees has already been described in the introduction. For binary increasing trees, it is very similar, except that every vertex can have at most two children: a left child, a right child, or both.

\subsection{Additive tree functionals}\label{sec:additive}

Many invariants associated with rooted trees satisfy a type of recursion in which the invariant is summed over all root branches. The notion of \emph{additive functionals} \cite{wagner2015central,Janson2016asymptotic, HolmgrenJanson, Ralaivaosaona2019central} provides a unifying framework for invariants of this kind. Let $F$ be an invariant that assigns a value $F(T)$ to every rooted tree $T$. This invariant is said to be \emph{additive} with \emph{toll function} $f$ if it satisfies the recursion 
\[F(T) = \sum_{i=1}^k F(T_i) + f(T)\]
for all rooted trees $T$, where $T_1,T_2,\ldots ,T_k$ are the root branches of $T$, i.e., the components obtained by removing the root from $T$, endowed with their natural roots (the children of $T$'s root).

\medskip

Many important examples of tree invariants satisfy such a recursion. A typical example is the number of leaves, whose toll function is
\[f(T) = \begin{cases} 1 & T \text{ is a single vertex,} \\ 0 & \text{otherwise.}\end{cases}\]
A generalisation of this example concerns \emph{fringe subtrees}: a fringe subtree of a rooted tree is a subtree consisting of a vertex and all its descendants. Thus a single leaf can also be regarded as a fringe subtree. The number of fringe subtrees of a given shape is also an additive functional in this sense, and this example will be important for us later. As we will see, the multiplicity of a fixed eigenvalue also fits this framework.

\medskip

There are limit theorems for additive tree functionals for several models of random trees and under different technical assumptions, see \cite{wagner2015central,Janson2016asymptotic, HolmgrenJanson, Ralaivaosaona2019central}. For our purposes, we particularly rely on a general central limit theorem due to Holmgren and Janson \cite{HolmgrenJanson}:
as usual, let $\mathcal{N}(\mu,\sigma^2)$ denote a normal distribution with mean $\mu$ and variance $\sigma^2$, and let $\xrightarrow{d}$ denote convergence in distribution. Moreover, let $\Rec_n$ and $\Bin_n$ denote a random recursive tree of order $n$ and a random binary increasing tree of order $n$ respectively.

\begin{thm}[{see \cite[Theorem~1.14]{HolmgrenJanson}}]\label{Thm:CLT}
\
	\begin{enumerate}
		\item For random recursive trees, assume that 
		\begin{align*}
			\sum _{k=1}^\infty \frac{\sqrt{\Var f(\Rec_k) }}{k^{3/2}}&<\infty,\\
			\lim _{k\rightarrow \infty} \frac{\Var f(\Rec_k) }{k} &=0,\\
			\sum _{k=1}^\infty \frac{(\Ex (f(\Rec_k)))^2 }{k^{2}} &<\infty.
		\end{align*}
		Then, as $n\rightarrow \infty,$
		\begin{align*}
			\frac{\Ex (F(\Rec_n))}{n}&\rightarrow \hat{\mu}_F := \sum _{k=1}^\infty \frac{1}{k(k+1)}\Ex (f(\Rec_k)),\\
			\frac{\Var (F(\Rec_n))}{n}&\rightarrow \hat{\sigma}^2_F < \infty,\\
		\end{align*}
		and \[\frac{F(\Rec_n)-\Ex (F(\Rec_n))}{\sqrt{n}}\xrightarrow{d}\mathcal{N}(0,\hat{\sigma}^2_F). \]
		\item For binary increasing trees, assume that 
		\begin{align*}
			\sum _{k=1}^\infty \frac{\sqrt{\Var (f(\Bin_k)) }}{k^{3/2}}&<\infty,\\
			\lim _{k\rightarrow \infty} \frac{\Var (f(\Bin_k)) }{k} &=0,\\
			\sum _{k=1}^\infty \frac{(\Ex (f(\Bin_k)))^2 }{k^{2}} &<\infty.
		\end{align*}
		Then, as $n\rightarrow \infty,$
		\begin{align*}
			\frac{\Ex (F(\Bin_n))}{n}&\rightarrow \mu_F := \sum _{k=1}^\infty \frac{2}{(k+1)(k+2)}\Ex (f(\Bin_k)),\\
			\frac{\Var (F(\Bin_n))}{n}&\rightarrow \sigma^2_F,
		\end{align*}
		and \[\frac{F(\Bin_n)-\Ex (F(\Bin_n))}{\sqrt{n}}\xrightarrow{d}\mathcal{N}(0,\sigma^2_F). \]
	\end{enumerate}
\end{thm}

\subsection{The toll function}\label{sec:toll}

In this section, we show that the multiplicity of an eigenvalue in the spectrum of a rooted tree can be viewed as an additive parameter with bounded toll function. A similar statement is also obtained for the Laplacian spectrum. We let $N_{\alpha}(T)$ denote the multiplicity of the eigenvalue $\alpha$ in the spectrum of a rooted tree $T.$ Let $T-r$ be the forest obtained from a rooted tree $T$ by deleting the root $r$.

Now we let $\lambda _1 \geq \lambda _2 \geq \cdots \geq \lambda _n$  and $\beta _1 \geq \beta _2 \geq \cdots \geq \beta _{n-1} $ be the eigenvalues of the tree $T$ and the forest $T-r$ respectively. Suppose that $\alpha = \lambda_{k+1}$ has multiplicity $l,$ that is 
\[\lambda _1 \geq \cdots \geq \lambda _k > \lambda_{k+1} = \cdots= \lambda_{k+l} > \lambda_{k+l+1}\geq \cdots \geq\lambda _n.\]
By the Cauchy interlacing theorem we have the following three possibilities.

\textbf{Case 1:}
\[\lambda_k\geq\beta_k>\lambda_{k+1} = \beta_{k+1}= \cdots= \beta_{k+l-1}= \lambda_{k+l} > \beta_{k+l}\geq\lambda_{k+l+1}\]
so that $ N_{\alpha}(T) - N_{\alpha} (T-r)= l- (l-1) = 1,$

\textbf{Case 2:}
\[\lambda_k>\beta_k=\lambda_{k+1} = \beta_{k+1}= \cdots= \beta_{k+l-1}= \lambda_{k+l} = \beta_{k+l}>\lambda_{k+l+1}\]
so that $ N_{\alpha}(T) - N_{\alpha} (T-r)= l- (l+1) = -1,$

\textbf{Case 3:}
\[\lambda_k>\beta_k=\lambda_{k+1} = \beta_{k+1}= \cdots= \beta_{k+l-1}= \lambda_{k+l} > \beta_{k+l}\geq\lambda_{k+l+1}\]
or 
\[\lambda_k\geq \beta_k>\lambda_{k+1} = \beta_{k+1}= \cdots= \beta_{k+l-1}= \lambda_{k+l} = \beta_{k+l}>\lambda_{k+l+1}\]
so that $ N_{\alpha}(T) - N_{\alpha} (T-r)= l-l = 0.$

So we see that $N_\alpha(T)-N_\alpha(T-r) \in \{-1,0,1\}$. The forest $T-r$ has the branches of $T$ as its components; let them be denoted by $T_1,T_2,\ldots,T_k$.
Since the multiplicity of an eigenvalue in the adjacency matrix of a graph can be obtained by summing the multiplicitity over the connected components (if we choose an appropriate vertex order, the adjacency matrix becomes a block diagonal matrix with the blocks corresponding to the components), we have
\[N_\alpha(T-r) = \sum _{i=1}^kN_\alpha(T_i).\]
 Altogether, we thus obtain the recursion
\[N_\alpha(T) = \sum _{i=1}^kN_\alpha(T_i) + n_\alpha(T),\]
where the toll function $n_\alpha(T)$ is given by \[n_\alpha(T)= N_\alpha(T)-N_\alpha(T-r) \in \{-1,0,1\}.\]
For the Laplacian spectrum, we can argue in a similar manner, but we have to modify the definition slightly. Instead of considering the multiplicity of $\alpha$ as a Laplacian eigenvalue of $T$, let $M_\alpha(T)$ be the multiplicity of $\alpha$ as an eigenvalue of the matrix obtained from the Laplacian matrix by adding $1$ to the diagonal entry in the row and column that correspond to the root (i.e., we artificially increase the root degree by $1$). Let this be called the \emph{modified Laplacian matrix}. As changing a single entry of a matrix can change its rank by at most $1$, $M_\alpha(T)$ differs from the multiplicity of $\alpha$ as an eigenvalue of the Laplacian matrix by at most $1$.

\medskip

If we remove the row and column that correspond to the root from the modified Laplacian matrix, then we obtain (up to reordering) a block diagonal matrix, where each block is the modified Laplacian matrix of one of the root branches. Here, it is important that we are using the modified Laplacian as opposed to the normal Laplacian matrix, since the degree of the root in each branch $T_i$ is exactly $1$ less than its degree in the whole tree $T$. Using the Cauchy interlacing theorem once again, we find that
\[M_\alpha(T) = \sum _{i=1}^kM_\alpha(T_i) + m_\alpha(T),\]
where $m_\alpha(T) \in \{-1,0,1\}$.

\section{The central limit theorem}\label{Sec:LimDist_Incr_tree}

From the discussion in the previous section, we know that $N_{\alpha}(T)$ can be regarded as an additive parameter whose toll function $n_{\alpha}(T)=N_{\alpha}(T)-N_{\alpha}(T-r)$ satisfies $n_{\alpha}(T)\in \{-1,0,1\}$. So in particular, if $\Rec_n$ and $\Bin_n$ denote a random recursive tree of order $n$ and a random binary increasing tree of order $n$ respectively,
\[\Ex (n_{\alpha}(\Rec_k)), \Var (n_{\alpha}(\Rec_k))\]
as well as
\[\Ex (n_{\alpha}(\Bin_k)), \Var (n_{\alpha}(\Bin_k))\]
are all trivially bounded by constants, showing that the technical conditions of Theorem~\ref{Thm:CLT}
are satisfied. Hence we have the following theorem:

\begin{thm}\label{thm:mult-CLT}
Fix a real number $\alpha$ that can occur as an eigenvalue of a tree. There exist constants $\mu_{\rec,\alpha}$ and $\sigma^2_{\rec,\alpha}$ such that the multiplicity $N_{\alpha}(\Rec_n)$ of $\alpha$ as an eigenvalue of the random recursive tree $\Rec_n$ with $n$ vertices satisfies
\[\frac{N_{\alpha}(\Rec_n)-\mu_{{\rec},\alpha}n}{\sqrt{n}}\xrightarrow{d}\mathcal{N}(0,\sigma^2_{{\rec},\alpha}). \]
Likewise, suppose that $\alpha$ can occur as an eigenvalue of a binary tree. There exist constants $\mu_{{\bin},\alpha}$ and $\sigma^2_{\bin,\alpha}$ such that the multiplicity $N_{\alpha}(\Bin_n)$ of $\alpha$ as an eigenvalue of the random binary increasing tree $\Bin_n$ with $n$ vertices satisfies
\[\frac{N_{\alpha}(\Bin_n)-\mu_{\bin,\alpha}n}{\sqrt{n}}\xrightarrow{d}\mathcal{N}(0,\sigma^2_{\bin,\alpha}). \]
\end{thm}

\begin{rem}
It is known that a real number $\alpha$ can be an eigenvalue of a tree if and only if it is a totally real algebraic integer, i.e., it is a zero of a monic polynomial whose coefficients are integers and whose zeros are all real (in other words, the conjugates of $\alpha$ must also all be real). It is clear that this is a necessary condition, and it was also shown to be sufficient by Salez \cite{Salez2015Totally}. Since eigenvalues of a graph are bounded by the maximum degree, not all such real numbers are also eigenvalues of a binary tree. We are not aware of any explicit characterisation of the possible eigenvalues of binary trees.
\end{rem}

A completely analogous statement holds for the Laplacian spectrum. Note here that we are applying Theorem~\ref{Thm:CLT} to the invariant $M_\alpha(T)$, which is not exactly the multiplicity of $\alpha$ as an eigenvalue of the Laplacian of $T$. However, since the difference is at most $1$, it vanishes in the limit due to the normalising factor $\sqrt{n}$.

\begin{thm}
Fix a real number $\alpha$ that can occur as an eigenvalue of the modified Laplacian matrix of a rooted tree. There exist constants $\nu_{\rec,\alpha}$ and $\tau^2_{\rec,\alpha}$ such that the multiplicity $L_{\alpha}(\Rec_n)$ of $\alpha$ as a Laplacian eigenvalue of the random recursive tree $\Rec_n$ with $n$ vertices satisfies
\[\frac{L_{\alpha}(\Rec_n)-\nu_{\rec,\alpha}n}{\sqrt{n}}\xrightarrow{d}\mathcal{N}(0,\tau^2_{\rec,\alpha}). \]
Likewise, there exist constants $\nu_{\bin,\alpha}$ and $\tau^2_{\bin,\alpha}$ such that the multiplicity $L_{\alpha}(\Bin_n)$ of $\alpha$ as a Laplacian eigenvalue of the random binary increasing tree $\Bin_n$ with $n$ vertices satisfies
\[\frac{L_{\alpha}(\Bin_n)-\nu_{\bin,\alpha}n}{\sqrt{n}}\xrightarrow{d}\mathcal{N}(0,\tau^2_{\bin,\alpha}). \]
\end{thm}

\begin{rem}
Any finite linear combination of multiplicities $N_{\alpha}$ can also be regarded as an additive functional whose toll function is still bounded. Therefore Theorem~\ref{Thm:CLT} still applies, which means that any finite linear combination $\sum_{i=1}^k c_i N_{\alpha_i}(T)$ also satisfies a central limit theorem. Applying the Cramér--Wold device \cite[Chapter 5, Theorem 10.5]{Gut2013probability}, we see that the multiplicities of any finite set of eigenvalues satisfy a multidimensional central limit theorem. An analogous statement also holds for Laplacian eigenvalues.
\end{rem}

\begin{rem}
Note that we can express the mean constants as
\[\mu_{\rec,\alpha} = \sum_{k=1}^{\infty} \frac{\Ex (n_{\alpha}(\Rec_k))}{k(k+1)} \text{ and } \mu_{\bin,\alpha} = \sum_{k=1}^{\infty} \frac{2\Ex (n_{\alpha}(\Bin_k))}{(k+1)(k+2)},\]
respectively. Analogous formulas hold for $\nu_{\rec,\alpha}$ and $\nu_{\bin,\alpha}$.

\medskip

It is clear that the sums
\[\sum_{\alpha} \mu_{\rec,\alpha} \text{ and } \sum_{\alpha} \mu_{\bin,\alpha}\]
and their counterparts for Laplacian eigenvalues are less than or equal to $1$. However, it is not possible to interchange the order of summation in
\[ \sum_{\alpha} \mu_{\rec,\alpha} = \sum_{\alpha} \sum_{k=1}^{\infty} \frac{\Ex (n_{\alpha}(\Rec_k))}{k(k+1)}\]
since the double sum may not be absolutely convergent. So it is not clear whether or not the sum is indeed equal to $1$, which would mean that the limiting spectral measure is purely discrete. The same comment applies to binary increasing trees and Laplacian eigenvalues.
\end{rem}

\subsection{Forcing subtrees and positivity of the mean constants}

Even if a real number $\alpha$ can occur as an eigenvalue of a tree, it is not a priori clear that the corresponding mean constant $\mu_{\rec}(\alpha)$ is not equal to $0$. In this section, we prove that this is always the case. In fact, we are able to provide lower bounds on $\mu_{\rec}(\alpha)$ and $\mu_{\bin}(\alpha)$. The main idea is to establish a relationship between the number of occurrences of a subtree $H$ in a large tree $T$ and specific eigenvalues in the spectrum of $T$. 

\medskip

Recall that a fringe subtree of a rooted tree $T$ is a subtree that consists of a vertex $v$ and all its descendants. Let us denote this subtree by $T_v$. It is easy to prove by induction that any additive functional $F$ with toll function $f$ satisfies
$$F(T) = \sum_v f(T_v)$$
if we define $f(\bullet) = F(\bullet)$ for the single-vertex tree $\bullet$. The number of occurrences of a fixed rooted tree $H$ as a fringe subtree (i.e., the number of fringe subtrees isomorphic to $H$) is an additive functional with a toll function given by $f(T) = 1$ whenever $T$ is isomorphic to $H$ and $f(T) = 0$ otherwise. Therefore, Theorem~\ref{Thm:CLT} applies to this functional (see also \cite{wagner2015central}), and we have the following:

\begin{prop}\label{prop:occurrences}
For any fixed rooted tree $H$, the number of occurrences of $H$ as a fringe subtree of a random recursive tree with $n$ vertices satisfies a central limit theorem with mean $\mu_{\rec,H}n + \Oh(1)$. The constant is given by $\mu_{\rec,H} = \frac{\beta_{\rec}(H)}{|H|(|H|+1)}$, where $\beta_{\rec}(H)$ is the probability that a random recursive tree with $|H|$ vertices is isomorphic to $H$.

Likewise, for any fixed rooted binary tree $H$, the number of occurrences of $H$ as a fringe subtree of a random binary increasing tree with $n$ vertices satisfies a central limit theorem with mean $\mu_{\bin,H}n + \Oh(1)$. The constant is given by $\mu_{\bin,H} = \frac{2\beta_{\bin}(H)}{(|H|+1)(|H|+2)}$, where $\beta_{\bin}(H)$ is the probability that a random binary increasing tree with $|H|$ vertices is isomorphic to $H$.
\end{prop}

Now, let us consider the relationship between the spectrum of a rooted tree $T$ and fringe trees found in $T$. When we join $k_i$ copies of a tree $H$ to different vertices $u_i$ ($i \in \{1,2,\ldots, l\}$) in a tree $T$, we can bound the multiplicities of certain eigenvalues in the resulting tree from below. This is captured in the following results.

\begin{thm}\label{Thm:DiffVecticesAdj}
	Let $T$ be a tree obtained from $G$ by joining $k_i$ copies of the rooted tree $H$ to the vertices $u_i\in V(G)$, $i \in \{1,2,\ldots, l\}$. That is, for each $i$, we take $k_i$ copies of $H$ and connect each of their roots to $u_i$ by an edge. Then each eigenvalue of $H$ is an eigenvalue of the resulting tree, and the multiplicity of each of these  eigenvalues is at least $\sum _{i=1}^l(k_i-1).$
\end{thm}

\begin{proof}
	Consider the forest $T \setminus \{u_1, u_2,\ldots, u_l\}.$ Obviously, it has $k_1+ k_2+\cdots +k_l$ components isomorphic to $H.$ So if $\alpha$ is an eigenvalue of $H,$ then it is an eigenvalue of  $T \setminus \{u_1, u_2,\ldots, u_l\}$ whose multiplicity is at least $k_1+ k_2+\cdots +k_l.$ The interlacing theorem shows that the multiplicity of $\alpha$ as an eigenvalue of $T$ differs from the multiplicity as an eigenvalue of $T \setminus \{u_1, u_2,\ldots, u_l\}$ by at most $l.$ Thus, $\alpha$ is an eigenvalue of $T$ with multiplicity at least
	\[k_1+ k_2+\cdots +k_l-l = (k_1-1)+(k_2-1)+\cdots +(k_l-1).\]
\end{proof}

As a simple corollary, we have the following known bound for the eigenvalue $0$. Let $l(T)$ be the number of leaves in the tree $T$, and let $q(T)$ be the number of quasipendant vertices (a vertex is called quasipendant if it is adjacent to a leaf) in the tree $T$.
\begin{cor}[{see e.g.~\cite[p.~258]{CDSSpectra}}]\label{Cor:leavesVsquasiAdj}
	The multiplicity of the eigenvalue $0$ in the adjacency spectrum of a tree $T$ is at least $l(T)-q(T).$
\end{cor}
\begin{proof}
This is the special case of Theorem~\ref{Thm:DiffVecticesAdj} where $H$ is a single vertex and $u_1,u_2,\ldots$ are the quasipendant vertices of $T$. If there are $l_i$ leaves attached to $u_i$, then by Theorem
\ref{Thm:DiffVecticesAdj}, the multiplicity of the eigenvalue 0 is at least 
	\[\sum _{i=1}^m(l_i-1)= \sum _{i=1}^ml_i- m =l(T)-q(T).\] 
\end{proof}

In particular, let $H$ be a rooted tree that has $\alpha$ as an eigenvalue, and let $K$ be obtained by joining two copies of $H$ to a common root. The presence of $K$ as a fringe subtree in $T$ ``forces'' $\alpha$ to be an eigenvalue of $T$. More generally, if $K$ occurs $r$ times as a fringe subtree of $T$, then Theorem~\ref{Thm:DiffVecticesAdj} shows that $\alpha$ is an eigenvalue of multiplicity at least $r$:
$N_{\alpha}(T) \geq r$. The following is now immediate from Proposition~\ref{prop:occurrences}:

\begin{cor}
Let $H$ be a rooted tree that has $\alpha$ as an eigenvalue, and let $K$ be obtained by joining two copies of $H$ to a common root. Then we have
$$\mu_{\rec,K} \leq \mu_{\rec,\alpha}.$$
\end{cor}

In particular, $\mu_{\rec,\alpha} > 0$ as soon as $\alpha$ is an eigenvalue of some tree. Analogous corollaries hold for binary increasing trees and for Laplacian eigenvalues. However, it is worth noting that the bound above is generally quite weak: while the presence of $K$ as a fringe subtree implies that $\alpha$ is an eigenvalue, the converse is not true. Moreover, there generally exist several non-isomorphic ``forcing subtrees'' like $K$ corresponding to any particular eigenvalue $\alpha$.

\subsection{Positivity of the variance constants}

We can use a similar argument to show that the variance constants in Theorem~\ref{thm:mult-CLT} are always strictly positive (in the general theorem of Holmgren and Janson, it is possible that the variance constants $\sigma_F$ and $\hat{\sigma}_F$ are $0$). We exhibit this for eigenvalues of recursive trees. The case of Laplacian eigenvalues is similar, as is the argument for binary increasing trees.

Let $\alpha$ be a fixed real number that is a possible eigenvalue of a tree. Using the ``forcing subtrees'' argument from the previous section, we can construct trees for which the multiplicity of $\alpha$ is arbitrarily large. In particular, we can find a rooted tree $S$ with $N_{\alpha}(S) \geq 4$. On the other hand, we let $S'$ be a path with the same number of vertices as $S$. Since paths do not have repeated eigenvalues (the eigenvalues of the $n$-vertex path are $2\cos \frac{\pi k}{n+1}$, $k \in \{1,2,\ldots,n\}$, see e.g.~\cite[p.~47]{DPSGraphSpectra}), we have $N_{\alpha}(S') \leq 1$.

Now let $T$ be an arbitrary tree, and suppose that $S$ occurs as a fringe subtree of $T$; in other words, $T$ can be obtained by attaching $S$ to a vertex $v$ of a rooted tree $R$. From arguments that we used before, it follows that
$$N_{\alpha}(T) \geq N_{\alpha}(R-v) + N_{\alpha}(S) - 1.$$
Let $T'$ be obtained by replacing $S$ with $S'$. By the same reasoning, we have
$$N_{\alpha}(T') \leq N_{\alpha}(R-v) + N_{\alpha}(S') + 1.$$
Combined with our assumptions on $N_{\alpha}(S)$ and $N_{\alpha}(S')$, this gives us
$$N_{\alpha}(T') \leq N_{\alpha}(T) -1.$$
Now let us consider a random recursive tree $\Rec_n$ with $n$ vertices, and let us replace all occurrences of either $S$ or $S'$ as a fringe subtree by a marked leaf to obtain a reduced tree $\Rec_n'$. As a consequence of Theorem~\ref{Thm:CLT} and Proposition~\ref{prop:occurrences}, the number of marked leaves in this reduced tree satisfies a central limit theorem with mean $(\mu_{\rec,S} + \mu_{\rec,S'})n +\Oh(1)$. Conditioned on its size, each fringe subtree of a random recursive tree is again a random recursive tree. Therefore, if we condition on the shape of the reduced tree $\Rec_n'$ obtained from a random recursive tree $\Rec_n$ and let the number of marked leaves in the reduced tree be $M_n$, then the number of fringe subtrees represented by marked leaves that are isomorphic to $S$ follows a binomial distribution $\operatorname{Bin}(M_n,p)$, where $p = \frac{\beta_{\rec,S}}{\beta_{\rec,S}+\beta_{\rec,S'}}$. Since replacing a fringe subtree isomorphic to $S$ by a fringe subtree isomorphic to $S'$ decreases the multiplicity of $\alpha$ as an eigenvalue by at least $1$, the variance of $N_{\alpha}(\Rec_n)$, conditioned on the value of $M_n$, is at least the variance of this binomial distribution, which is $p(1-p)M_n$. It now follows from the law of total variance that
\begin{align*}
    \Var (N_{\alpha}(\Rec_n)) &= \Ex (\Var(N_{\alpha}(\Rec_n)|M_n)) + \Var(\Ex (N_{\alpha}(\Rec_n)|M_n)) \\
    &\geq \Ex (\Var(N_{\alpha}(\Rec_n)|M_n)) \\
    &\geq \Ex(p(1-p)M_n) \\
    &= p(1-p)(\mu_{\rec,S} + \mu_{\rec,S'})n +\Oh(1),
\end{align*}
which implies that $\sigma^2_{\rec,\alpha} \geq p(1-p)(\mu_{\rec,S} + \mu_{\rec,S'}) > 0$. This shows that the variance constant is always positive, so that the multiplicity of $\alpha$ as an eigenvalue, suitably normalised, converges weakly to a non-degenerate normal distribution.

\section{The special case of the eigenvalue $0$}\label{sec:EV0}

The eigenvalue $0$ of the adjacency matrix is special for several reasons. Since the spectrum of every tree is symmetric, the multiplicity of $0$ as an eigenvalue is even if and only if the number of vertices is. This implies that the toll function $n_0$ defined in Section~\ref{sec:toll} can only take the values $-1$ and $1$, as opposed to the generic case where $n_{\alpha}$ can also attain the value $0$.

\medskip

One can also recursively characterise the toll function $n_0$, which will allow us to explicitly determine the mean constants $\mu_{\rec,0}$ and $\mu_{\bin,0}$ as well as the variance constants $\sigma^2_{\rec,0}$ and $\sigma^2_{\bin,0}$ using methods from analytic combinatorics.

\medskip

To this end, let $T$ be a rooted tree with root $r$, and let $\Psi (T,z)$  and $\Psi _r (T,z)$ be the characteristic polynomials of $T$ and $T-r$ respectively. Note that
$n_0(T) = 1$ if and only if $\lim_{z \to 0^+} \frac{\Psi_r(T,z)}{\Psi(T,z)} = \pm \infty$, and $n_0(T) = -1$ if and only if $\lim_{z \to 0^+} \frac{\Psi_r(T,z)}{\Psi(T,z)} = 0$.

\medskip

The ratio $\frac{\Psi_r(T,z)}{\Psi(T,z)}$ can be related to the \emph{angles} of $T$. Suppose that $\{\e_1 ,\e_2 , \ldots , \e_n\}$ are the natural basis vectors, that $\lambda_1,\lambda_2,\ldots,\lambda_m$ are the distinct eigenvalues of $T$, and that the adjacency matrix $A(T)$ has the spectral decomposition
\[A(T) = \lambda_1 P_1 + \lambda_2 P_2 + \cdots + \lambda_m P_m,\]
where $P_i$ represents the orthogonal projection of $\mathbb{R}^n$ onto the eigenspace of $\lambda_i$ with respect to the basis $\{\e_1 ,\e_2 , \ldots , \e_n\}$. The values of $\theta_{ij} =||P_ie_j||$, where $1\leq i\leq m$ and $1\leq j\leq n,$ are called the angles of $T$. If the $j$-th column corresponds to a specific vertex $v$, then we say that $\theta_{ij}$ are the angles corresponding to $v$.

\begin{thm}[{see \cite[p.~33]{DPSGraphSpectra}}]\label{Thm:Ratio_CharPoly_angles}
	Let $\Psi (T,z)$  and $\Psi_r (T,z)$ be the characteristic polynomials of a rooted tree $T$ with root $r$ and the forest $T-r$ respectively. Suppose that $\lambda_1, \lambda_2, \ldots \lambda_m$ are the distinct eigenvalues of $T$ and $\theta_{ir}$ are the angles corresponding to $r$. Then 
	\[\frac{\Psi _r(T,z)}{\Psi (T,z)}= \sum_{i=1}^m  \frac{\theta _{ir}^2}{z-\lambda _i}. \]
\end{thm}

This implies in particular that $\lim_{z \to 0^+} \frac{\Psi _r(T,z)}{\Psi (T,z)}$ cannot be $-\infty$, since all numerators are non-negative. The limit is therefore either $0$ or $+\infty$.

\medskip

Now we apply the following recursion for the characteristic polynomial due to Mohar \cite{BMohar89}:

\begin{thm}[see~\cite{BMohar89}]\label{Thm:Ratio_CharPoly}
	Let $\Psi (T,z)$  and $\Psi _r (T,z)$ be the characteristic polynomials of a tree $T$ and the forest $T-r$ respectively, where $r$ is the root of $T$. Suppose $T_1,T_2, \ldots ,T_k$ are the branches of $T$ with $v_1, v_2, \ldots , v_k$ their respective roots. Then 
	\begin{align}
	\frac{\Psi _r(T,z)}{\Psi (T,z)} = \frac{1}{z - \sum_{j=1}^k \frac{\Psi _{v_j}(T_j,z)}{\Psi (T_j,z)}}. \label{Eqn:Ratio_CharPoly_in_terms_of_eig_product}
	\end{align}
\end{thm}

This readily shows that if $\lim_{z \to 0^+} \frac{\Psi _{v_j}(T_j,z)}{\Psi (T_j,z)} = 0$ for all $j$, we have $\lim_{z \to 0^+} \frac{\Psi_r(T,z)}{\Psi (T,z)} = + \infty$. In other words, if $n_0(T_j) = -1$ for all $j$, then $n_0(T) = 1$. On the other hand, if at least one of the fractions $\frac{\Psi_r(T,z)}{\Psi (T,z)}$ tends to $+\infty$ as $z \to 0^+$ (equivalently, if $n_0(T_j) = 1$ for at least one index $j$), then the denominator in~\eqref{Eqn:Ratio_CharPoly_in_terms_of_eig_product} tends to $-\infty$, and we have $\lim_{z \to 0^+} \frac{\Psi_r(T,z)}{\Psi (T,z)} = 0$, thus $n_0(T) = -1$.

\medskip

In summary,
\begin{equation}\label{eq:n0rec}
n_0(T) = \begin{cases} 1 & \text{if } n_0(T_j) = -1 \text{ for all $j$,} \\
-1 & \text{if } n_0(T_j) = 1 \text{ for at least one $j$.} \end{cases}
\end{equation}
This recursion will be translated to functional equations for generating functions in the following sections.

\subsection{Recursive trees}

In this section, we focus on the multiplicity of $0$ in the spectrum of recursive trees. As indicated earlier, recursive trees are increasing trees that can be constructed by a growth or evolution process, where each new vertex is attached to one of the previous vertices uniformly at random. As there are $n-1$ possibilities when the $n$-th vertex is attached, there are $(n-1)!$ different recursive trees.

\medskip

By marking the root of a recursive tree with the lowest label, a recursive tree can symbolically be defined by a boxed product. That is, a marked node attached to a set of recursive trees. If we let $Y(x)$ be the exponential generating function, this symbolic definition translates to the differential equation
\[Y'(x) = \exp(Y(x)), \quad \text{where  } Y(0)=0,\]
see Section~\ref{sec:increasing}. This differential equation has the explicit solution $Y(x) = -\log(1-x)$,
which is consistent with the earlier observation that there are $(n-1)!$ recursive trees with $n$ vertices.

We now define a bivariate exponential generating function $Y(x,t)$ that takes into account the multiplicity of the eigenvalue 0:
\[Y(x,t) = \sum _{T \in \RT} e^{tN_0(T)}\frac{x^{|T|}}{|T|!},\] where the sum is over the set $\RT$ of all recursive trees. Of course, $Y(x,0) = Y(x) = -\log (1-x)$.

Based on the recursive characterisation of $n_0$ in~\eqref{eq:n0rec}, we divide the set of all recursive trees into the sets $\RT_+$ and $\RT_-$ of trees satisfing $n_0(T) = 1$ and $n_0(T) = -1$, respectively. We also let $Y_+(x,t)$ and $Y_-(x,t)$ be the respective associated exponential generating functions, defined in the same way as $Y(x,t)$ above. The recursion~\eqref{eq:n0rec} translates to the system of equations
\begin{align*}
\frac{\partial}{\partial x} Y_+(x,t) &= \exp\big(t + Y_{-}(x,t) \big), \\
\frac{\partial}{\partial x} Y_-(x,y) &= \exp\big({-t} + Y(x,t)\big) - \exp\big({-t} + Y_-(x,t)\big) \\
&= \exp\big({-t} + Y_-(x,t)\big) \big( \exp \big( Y_+(x,t) \big) - 1 \big).
\end{align*}
This system of partial differential equations does not seem to have an explicit solution. However, in order to determine mean and variance of $N_0(\Rec_n)$, we only need the coefficients of $x^n$ in the partial derivatives $\frac{\partial}{\partial t} Y(x,t) \Big|_{t=0}$ and $\Big( \frac{\partial}{\partial t} \Big)^2 Y(x,t) \Big|_{t=0}$, and these satisfy ordinary differential equations that we will actually be able to solve.

In order to simplify some of the equations that follow, we introduce the abbreviations $F(x) = Y(x,0) = - \log(1-x)$, $F_{\pm}(x) = Y_{\pm}(x,0)$ as well as \[G(x) = \frac{\partial}{\partial t} Y(x,t) \Big|_{t=0},\ G_{\pm}(x) = \frac{\partial}{\partial t} Y_{\pm}(x,t) \Big|_{t=0}\] and \[H(x) = \Big( \frac{\partial}{\partial t} \Big)^2 Y(x,t) \Big|_{t=0},\ H_{\pm}(x) = \Big( \frac{\partial}{\partial t} \Big)^2 Y_{\pm}(x,t) \Big|_{t=0}.\]
Let us start by determining $F_+(x)$ and $F_-(x)$. Setting $t=0$ in the system of differential equations gives us
\begin{align*}
F'_+(x) &= \exp\big(F_-(x) \big), \\
F'_-(x) &= \exp\big(F_-(x)\big) \big( \exp \big( F_-(x) \big) - 1 \big).
\end{align*}
Since we know that $F_-(x) = F(x) - F_+(x) = -\log(1-x) - F_+(x)$, the first of these equations yields
\[F'_+(x) = \frac{\exp(-F_+(x))}{1-x}.\]
This separable differential equation with initial condition $F_+(0) = 0$ has the explicit solution
\[
F_+(x) = \log(1 - \log(1-x)),
\]
and we also get
\[
F_-(x) = F(x) - F_+(x) = -\log \big( (1-x)(1-\log(1-x))\big).
\]
Next we consider the first derivatives. Differentiating the system of differential equations for $Y_+(x,t)$ and $Y_-(x,t)$ with respect to $t$ and plugging in $t=0$, we obtain
\begin{align*}
G'_+(x) &= \exp\big(F_{-}(x) \big) \big( 1 + G_{-}(x) \big), \\
G'_-(x) &= \exp\big(F_{-}(x) + F_{+}(x)\big) \big({-1} + G_{-}(x) + G_{+}(x) \big) - \exp\big(F_{-}(x) \big) \big({-1} + G_{-}(x) \big).
\end{align*}
Adding these two and plugging in the formulas for $F(x)$ and $F_-(x)$, we obtain
\[G'(x) = \exp(F(x)) (G(x) - 1) + 2 \exp \big(F_{-}(x) \big) = \frac{G(x)}{1-x} - \frac{1}{1-x} + \frac{2}{(1-x)(1-\log(1-x))}.\]
Multiplying by $1-x$ and regrouping gives
\[\frac{d}{dx} \big( (1-x)G(x) \big) = (1-x) G'(x) - G(x) = \frac{1+ \log(1-x)}{1-\log(1-x)},\]
hence we have (taking the initial condition $G(0) = 0$ into account)
\[G(x) = \frac{1}{1-x} \int_0^x \frac{1+\log(1-u)}{1-\log(1-u)}\,du.\]
At this point, let us introduce the function
\[A(x) = \int_0^x \frac{1}{1-\log(1-u)} \,du\]
and note some of its properties: we have $A(0) = 0$, $A(1) = G$, where $G$ denotes the Euler--Gompertz constant \cite[pp.~425--426]{Finch2002mathematical}, as well as $A'(x) = \frac{1}{1-\log(1-x)}$. For later use, we also need its asymptotic behaviour as $x \to 1$. Integration by parts yields
\begin{align*}
A(x) &= G - \int_x^1 \frac{1}{1-\log(1-u)} \,du = G - \int_0^{1-x} \frac{1}{1-\log u}\,du \\
&= G - \frac{u}{1-\log u} \Big|_0^{1-x} + \int_0^{1-x}\frac{1}{(1-\log u)^2}\,du \\
&= G - \frac{1-x}{1-\log(1-x)} + \int_0^{1-x}\frac{1}{(1-\log u)^2}\,du.
\end{align*}
So it follows that
\[A(x) = G + \frac{1-x}{\log(1-x)} + \Oh \Big( \Big| \frac{1-x}{\log^2(1-x)} \Big| \Big)\]
as $x \to 1$. In terms of this function, we can express $G(x)$ as
\[G(x) = \frac{1}{1-x} \int_0^x \frac{1+\log(1-u)}{1-\log(1-u)}\,du = \frac{1}{1-x} \int_0^x \Big( \frac{2}{1-\log(1-u)} - 1 \Big) \,du = \frac{2A(x) - x}{1-x}.\]
We will also need $G_-(x)$ later, so let us return to the differential equation
\[G'_+(x) = \exp\big(F_{-}(x) \big) \big( 1 + G_{-}(x) \big),\]
which is equivalent to
\begin{align*}
G'_-(x) &= G'(x) - G'_+(x) = \frac{2A'(x)-1}{1-x} + \frac{2A(x)-x}{(1-x)^2} - \frac{1}{(1-x)(1-\log(1-x))} \big(1 + G_{-}(x) \big) \\
&= \frac{2A(x)-1}{(1-x)^2} + \frac{1}{(1-x)(1-\log(1-x))} \big(1 - G_{-}(x) \big).
\end{align*}
Multiplying by the integrating factor $1-\log(1-x)$ and rearranging yields
\[\frac{d}{dx} \big((1-\log(1-x))G_{-}(x)\big) = \frac{(2A(x)-1)(1-\log(1-x))}{(1-x)^2} + \frac{1}{1-x}.\]
We integrate both sides and apply integration by parts to obtain
\begin{align*}
(1-\log(1-x))G_{-}(x) &= -\frac{(2A(t)-1)\log(1-t)}{1-t} \Big|_0^x + \int_0^x \frac{2A'(u)\log(1-u)}{1-u}\,du - \log(1-x)\\
&= -\frac{(2A(x)-1)\log(1-x)}{1-x} + \log(1-x) + 2\log(1-\log(1-x))
\end{align*}
and finally
\begin{equation}\label{eq:G-}
G_{-}(x) = -\frac{(2A(x)-1)\log(1-x)}{(1-x)(1-\log(1-x))} + \frac{\log(1-x) + 2\log(1-\log(1-x))}{(1-\log(1-x))}.
\end{equation}
Now we have all ingredients to determine $H(x) = \Big( \frac{\partial}{\partial t} \Big)^2 Y(x,t) \Big|_{t=0}$ as well. We differentiate the system of differential equations for $Y_{\pm}(x,t)$ twice with respect to $t$, then plug in $t=0$ and add the two equations. After some simplifications, this gives us
\begin{align*}
H'(x) &= e^{F(x)} \big(H(x) + (G(x)-1)^2 \big) + 4 e^{F_-(x)}G_{-}(x) \\
&= \frac{H(x)}{1-x} + \frac{(2A(x)-1)^2}{(1-x)^3} + \frac{4G_{-}(x)}{(1-x)(1-\log(1-x))}.
\end{align*}
Again, we can solve this differential equation by multiplying by the integrating factor $1-x$ and integrating, which finally gives us the expression
\[H(x) = \frac{1}{1-x} \int_0^x \Big( \frac{(2A(u)-1)^2}{(1-u)^2} + \frac{4G_{-}(u)}{1-\log(1-u)} \Big)\,du.\]
Using~\eqref{eq:G-} and the properties of the function $A(x)$ mentioned previously, we find that
\[\frac{(2A(u)-1)^2}{(1-u)^2} + \frac{4G_{-}(u)}{1-\log(1-u)} = \frac{(2G-1)^2}{(1-u)^2} + \Oh \Big( \Big| \frac{1}{(1-u)\log^2(1-u)} \Big| \Big)\]
as $u \to 1$. This means in particular that the integral
\begin{equation}\label{eq:K1}
K_1 = \int_0^1 \Big( \frac{(2A(u)-1)^2}{(1-u)^2} + \frac{4G_{-}(u)}{1-\log(1-u)} - \frac{(2G-1)^2}{(1-u)^2} \Big)\,du
\end{equation}
is convergent, and we can write
\begin{align*}
H(x) &= \frac{(2G-1)^2x}{(1-x)^2} + \frac{1}{1-x} \int_0^x \Big( \frac{(2A(u)-1)^2}{(1-u)^2} + \frac{4G_{-}(u)}{1-\log(1-u)} - \frac{(2G-1)^2}{(1-u)^2} \Big)\,du \\
&= \frac{(2G-1)^2x}{(1-x)^2} + \frac{K_1}{1-x} + \frac{1}{1-x} \int_x^1 \Big( \frac{(2A(u)-1)^2}{(1-u)^2} + \frac{4G_{-}(u)}{1-\log(1-u)} - \frac{(2G-1)^2}{(1-u)^2} \Big)\,du.
\end{align*}
We are now in a position to determine the asymptotic behaviour of the coefficients by means of singularity analysis \cite[Chapter VI]{FlajoletRobert}. All functions we have been dealing with have a dominant singularity at $1$ (i.e., there is no other singularity whose modulus is less than or equal to $1$). Due to the closure properties of functions amenable to singularity analysis (see \cite[Section VI.10]{FlajoletRobert}), we can apply singularity analysis to all of them. Specifically, the asymptotic expansion of $A(x)$ translates to an expansion for $G(x)$:
\[G(x) = \frac{2A(x) - x}{1-x} = \frac{2G-1}{1-x} + 1 + \frac{1}{\log(1-x)} + \Oh \Big( \Big| \frac{1}{\log^2(1-x)} \Big| \Big).\]
This translates to an asymptotic formula for the coefficients of $G(x)$ and thus the mean of $N_0$. We have
\[\Ex (N_0(\Rec_n)) = \frac{[x^n] G(x)}{[x^n] F(x)} = n [x^n] G(x) = (2G-1)n + \Oh \Big( \frac{1}{\log n} \Big).\]
Let us state this as an explicit proposition.

\begin{prop}
The mean multiplicity of the eigenvalue $0$ of a random recursive tree with $n$ vertices is
\[\Ex (N_0(\Rec_n)) = (2G-1)n + \Oh \Big( \frac{1}{\log n} \Big).\]
\end{prop}
Since $2G-1 \approx 0.192694$, we infer that approximately $19.3\%$ of the spectrum of a large random recursive tree consist of the eigenvalue 0. In the same way, we can also deal with the variance. Since
\[H(x) = \frac{(2G-1)^2x}{(1-x)^2} + \frac{K_1}{1-x} + \Oh \Big( \Big| \frac{1}{(1-x)\log^2(1-x)} \Big| \Big),\]
we have
\[\Ex (N_0^2(\Rec_n)) = \frac{[x^n] H(x)}{[x^n] F(x)} = n [x^n] H(x) = (2G-1)^2n^2 + K_1n + \Oh \Big( \frac{n}{\log^2 n} \Big).\]
So we obtain the following final result on the variance.

\begin{prop}
The variance of the multiplicity of the eigenvalue $0$ of a random recursive tree with $n$ vertices is
\[\Var (N_0(\Rec_n)) = K_1n + \Oh \Big( \frac{n}{\log n} \Big),\]
with $K_1$ as defined in~\eqref{eq:K1}.
\end{prop}

The integral representation for $K_1$ can be simplified using integration by parts. There are many different ways to represent $K_1$, for example
\begin{align*}
    K_1 &= 4G(G-1) + 8 \int_0^1 \frac{\log(1-\log w)}{(1-\log w)^2} dw \\
    &= 4(G-1)(G-2) - 4\int_0^{\infty} \log^2(1+u) e^{-u}\,du \\
	&\approx 0.138629.
\end{align*}

\subsection{Binary increasing trees}

Now we consider the multiplicity of $0$ in the spectrum of binary increasing trees. The approach is very similar to recursive trees.

A binary increasing tree has the property that each vertex has 2 possible places to which a child can be attached, that is, a left or right child. In view of that, to attach a node labelled $n$ to an existing binary increasing tree of order $n-1$, there are $n$ possible places to do so. It follows that there are $n!$ different binary increasing trees with $n$ vertices. 

Let $\BT$ be the set of all binary increasing trees. As in the previous subsection, we denote the exponential generating function by $Y(x)$. In view of the recursive decomposition into root, left subtree and right subtree, we have
\[Y'(x) = (1+Y(x))^2, \quad \text{where  } Y(0)=0,\]
see again Section~\ref{sec:increasing}. This differential equation has the explicit solution $Y(x) = \frac{x}{1-x}$, in agreement with our earlier observation that there are $n!$ binary increasing trees with $n$ vertices.

As in the case of recursive trees, we can now define a bivariate exponential generating function $Y(x,t)$ in the following way:
\[Y(x,t) = \sum _{T \in \BT} e^{tN_0(T)}\frac{x^{|T|}}{|T|!},\] where the sum is over the set $\BT$ of all binary increasing  trees. Of course, we have $Y(x,0) = Y(x) = \frac{x}{1-x}$.

Now we exploit the recursive characterisation of $n_0$ in~\eqref{eq:n0rec} again, by dividing the set of all binary increasing trees into the sets $\BT_+$ and $\BT_-$ of trees satisfing $n_0(T) = 1$ and $n_0(T) = -1$, respectively. We also let $Y_+(x,t)$ and $Y_-(x,t)$ be the respective associated exponential generating functions again. The recursion~\eqref{eq:n0rec} now gives us the system of equations
\begin{align*}
\frac{\partial}{\partial x} Y_+(x,t) &= e^t \big(1 + Y_{-}(x,t) \big)^2, \\
\frac{\partial}{\partial x} Y_-(x,y) &= e^{-t} \big(1+ Y(x,t)\big)^2 - e^{-t} \big(1 + Y_-(x,t)\big)^2.
\end{align*}
Again, we cannot find an explicit solution for the bivariate generating function, but we can determine the partial derivatives at $0$. We use the same abbreviations as in the previous section, namely $F(x) = Y(x,0) = \frac{x}{1-x}$, $F_{\pm}(x) = Y_{\pm}(x,0)$ as well as \[G(x) = \frac{\partial}{\partial t} Y(x,t) \Big|_{t=0},\ G_{\pm}(x) = \frac{\partial}{\partial t} Y_{\pm}(x,t) \Big|_{t=0}\] and \[H(x) = \Big( \frac{\partial}{\partial t} \Big)^2 Y(x,t) \Big|_{t=0},\ H_{\pm}(x) = \Big( \frac{\partial}{\partial t} \Big)^2 Y_{\pm}(x,t) \Big|_{t=0}.\]
We start again by determining $F_+(x)$ and $F_-(x)$. Setting $t=0$ in the system of differential equations, we obtain
\begin{align*}
F'_+(x) &= \big(1 + F_-(x) \big)^2, \\
F'_-(x) &= \big(1 + F(x) \big)^2 - \big(1 + F_-(x) \big)^2.
\end{align*}
Now $F_-(x) = F(x) - F_+(x) = \frac{x}{1-x} - F_+(x)$, thus
\[F'_+(x) = \Big( \frac{1}{1-x} - F_+(x) \Big)^2.\]
This Riccati-type differential equation (with the initial value $F_+(0) = 0$) has the explicit solution
\[F_+(x) = \frac{3-\sqrt{5}}{2(1-x)} + \frac{2\sqrt{5}}{(1-x)(2-(7+3\sqrt{5})(1-x)^{-\sqrt{5}})}.\]
Consequently, since $F_+(x)+F_-(x) = F(x) = \frac{x}{1-x}$,
\[F_-(x) = \frac{\sqrt{5}-1}{2(1-x)} - \frac{2\sqrt{5}}{(1-x)(2-(7+3\sqrt{5})(1-x)^{-\sqrt{5}})} - 1.\]
Now we take the first derivative with respect to $t$ again. Differentiating the system of differential equations and plugging in $t=0$ gives
\begin{align*}
G_+'(x) &= F_+'(x) + 2(1+F_-(x))G_-(x), \\
G_-'(x) &= -F_-'(x) + 2(1+F(x))G(x) - 2(1+F_-(x))G_-(x).
\end{align*}
Adding the two yields
\[G'(x) = F_+'(x) - F_-'(x) + 2(1+F(x))G(x).\]
Since $1+F(x) = \frac{1}{1-x}$, we can solve this differential equation by multiplying by the integrating factor $(1-x)^2$ and integrating, which results in
\begin{align*}
G(x) &= \frac{1}{(1-x)^2} \int_0^x (1-u)^2 \big( F_+'(u) - F_-'(u) \big) \,du \\
&= \frac{1}{(1-x)^2} \int_0^x \frac{(8+4\sqrt{5})(1-u)^{2\sqrt{5}}+(84+36\sqrt{5})(1-u)^{\sqrt{5}}-(22+10\sqrt{5})}{(7+3\sqrt{5}-2(1-u)^{\sqrt{5}})^2} \,du \\
&= \frac{1}{(1-x)^2} \int_0^1 \frac{(8+4\sqrt{5})v^{2\sqrt{5}}+(84+36\sqrt{5})v^{\sqrt{5}}-(22+10\sqrt{5})}{(7+3\sqrt{5}-2v^{\sqrt{5}})^2} \,dv \\
&\quad - \frac{1}{(1-x)^2} \int_0^{1-x} \frac{(8+4\sqrt{5})v^{2\sqrt{5}}+(84+36\sqrt{5})v^{\sqrt{5}}-(22+10\sqrt{5})}{(7+3\sqrt{5}-2v^{\sqrt{5}})^2} \,dv.
\end{align*}
The integrand has the expansion
\[(2-\sqrt{5}) + (50-22\sqrt{5})v^{\sqrt{5}} + \Oh \big(|v|^{2\sqrt{5}}\big),\]
so it follows immediately that
\[G(x) = \frac{C_1}{(1-x)^2} - \frac{2-\sqrt{5}}{1-x} + \Oh \big( |1-x|^{\sqrt{5}-1} \big),\]
where
\[C_1 = \int_0^1 \frac{(8+4\sqrt{5})v^{2\sqrt{5}}+(84+36\sqrt{5})v^{\sqrt{5}}-(22+10\sqrt{5})}{(7+3\sqrt{5}-2v^{\sqrt{5}})^2} \,dv.\]
Singularity analysis now shows that
\[\Ex(N_0(\Bin_n)) = \frac{[x^n] G(x)}{[x^n] F(x)} = C_1(n+1) + \sqrt{5}-2 + \Oh \big( n^{-\sqrt{5}} \big).\]
Let us state this explicitly:
\begin{prop}
The mean multiplicity of the eigenvalue $0$ of a random binary increasing tree with $n$ vertices is
\[\Ex (N_0(\Bin_n)) = C_1(n+1) + \sqrt{5}-2 + \Oh \Big( n^{-\sqrt{5}} \Big),\]
where $C_1 \approx 0.085753$.
\end{prop}
So approximately $8.6\%$ of the spectrum of a large random binary increasing tree consist of the eigenvalue 0. It remains to consider the variance. For this purpose, we will need $G_-(x)$. Recall that $G_-(x)$ satisfies the linear differential equation
\begin{align*}
G_-'(x) &= -F_-'(x) + 2(1+F(x))G(x) - 2(1+F_-(x))G_-(x) \\
&= -F_-'(x) + \frac{2G(x)}{1-x} - 2(1+F_-(x))G_-(x).
\end{align*}
We can solve it by multiplying by the integrating factor
\[\exp \Big( \int 2(1+F_-(x))\,dx \Big) = \frac{\big(7+3\sqrt{5}-2(1-x)^{\sqrt{5}}\big)^2}{(1-x)^{\sqrt{5}-1}}\]
and integrating both sides. This gives us
\begin{equation}\label{eq:bin-G-}
G_-(x) = \frac{(1-x)^{\sqrt{5}-1}}{\big(7+3\sqrt{5}-2(1-x)^{\sqrt{5}}\big)^2} \int_0^x \frac{\big(7+3\sqrt{5}-2(1-u)^{\sqrt{5}}\big)^2}{(1-u)^{\sqrt{5}-1}} \Big(\frac{2G(u)}{1-u} -F_-'(u)\Big)\,du.
\end{equation}
Lastly, $H(x)$ is determined by first differentiating the system of equations for $Y_+(x,t)$ and $Y_-(x,t)$ twice with respect to $t$, plugging in $t=0$ and adding the equations, which yields
\[H'(x) = (1+F(x))^2 - 4(1+F(x))G(x) + 2G(x)^2 + 8(1+F_-(x)) G_-(x) + 2(1+F(x))H(x).\]
Here, the integrating factor is simply $(1-x)^2$ as in the differential equation for $G(x)$. This eventually gives us
\[H(x) = \frac{1}{(1-x)^2} \int_0^x \big(1-4(1-u)G(u) + 2(1-u)^2G(u)^2 + 8(1-u)^2(1 + F_-(u)) G_-(u)\big)\,du.\]
In order to apply singularity analysis, we need the expansion around the dominant singularity $x=1$. Recall that
\[G(x) = \frac{C_1}{(1-x)^2} - \frac{2-\sqrt{5}}{1-x} + \Oh \big( |1-x|^{\sqrt{5}-1} \big).\]
Moreover,
\[F_-(x) = \frac{\sqrt{5}-1}{2(1-x)} -1 + \Oh \big( |1-x|^{\sqrt{5}-1} \big).\]
Plugging these into the representation~\eqref{eq:bin-G-} for $G_-(x)$, we also find that
\[G_-(x) = \frac{C_1(\sqrt{5}-1)}{2(1-x)^2} - \frac{7-3\sqrt{5}}{2\sqrt{5}(1-x)} + \Oh(1).\]
We can now plug the asymptotic expansions of $G,F_-$ and $G_-$ into the integral representation for $H$, which shows that the integrand is
\[1-4(1-u)G(u) + 2(1-u)^2G(u)^2 + 8(1-u)^2(1 + F_-(u)) G_-(u) =  \frac{2C_1^2}{(1-u)^2} + \Oh(1).\]
In particular, it follows that the integral
\[C_2 = \int_0^1 \Big( 1-4(1-u)G(u) + 2(1-u)^2G(u)^2 + 8(1-u)^2(1 + F_-(u)) G_-(u) - \frac{2C_1^2}{(1-u)^2} \Big) \,du\]
is convergent, and we obtain
\[H(x) = \frac{2C_1^2x}{(1-x)^3} + \frac{C_2}{(1-x)^2} + \Oh \Big( \frac{1}{|1-x|} \Big). \]
Singularity analysis gives us the asymptotic behaviour of the second moment:
\[\Ex(N_0^2(\Bin_n)) = \frac{[x^n] H(x)}{[x^n] F(x)} = C_1^2n(n+1) + C_2n + \Oh(1),\]
so finally
\[\Var(N_0(\Bin_n)) = (4C_1-2\sqrt{5}C_1-C_1^2+C_2)n + \Oh(1).\]

\begin{prop}
The variance of the multiplicity of the eigenvalue $0$ of a random binary increasing tree with $n$ vertices is
\[\Var (N_0(\Bin_n)) = K_2n + \Oh(1),\]
with $K_2 = (4C_1-2\sqrt{5}C_1-C_1^2+C_2) \approx 0.057162$.
\end{prop}

\subsection{Connections to other graph parameters}

It is well known that the characteristic polynomial of the adjacency matrix of a forest coincides with its matching polynomial (a special case of a general result due to Sachs, see~\cite[Theorem 1.3]{CDSSpectra}), which in turn implies that the multiplicity of $0$ as an eigenvalue is given by $N_0(T) = |T|-2m(T)$, where $m(T)$ is the matching number (the greatest cardinality of a matching) of $T$. Moreover, it is a known consequence of König's theorem that independence number (i.e., the greatest cardinality of an independent set) and matching number of any bipartite graph (thus in particular any tree) add up to the order of the graph. Consequently, we can express the independence number $i(T)$ of a tree $T$ in terms of the multiplicity of $0$ as well: $i(T) = |T|-m(T) = \frac12 (N_0(T) + |T|)$.

Therefore, we obtain central limit theorems for independence number and matching number as simple corollaries: with the exception of the variance constants, these have been proven before \cite{Fuchs2021note,Janson2020independence}, and the mean for recursive trees has been determined even earlier \cite{Meir1975expected}.

\begin{cor}
The independence number and matching number of a random recursive tree with $n$ vertices satisfy the following central limit theorem:
$$\frac{i(\Rec_n) - c_{\rec} n}{s_{\rec}\sqrt{n}} \xrightarrow{d} \mathcal{N}(0,1) \text{ and } \frac{m(\Rec_n) - (1-c_{\rec}) n}{s_{\rec}\sqrt{n}} \xrightarrow{d} \mathcal{N}(0,1),$$
where $c_{\rec} = G \approx 0.596347$ and $s_{\rec} = \frac{K_1}{2} \approx 0.069315$.

Likewise, the independence number and matching number of a random binary increasing tree with $n$ vertices satisfy the following central limit theorem:
$$\frac{i(\Bin_n) - c_{\bin} n}{s_{\bin}\sqrt{n}} \xrightarrow{d} \mathcal{N}(0,1) \text{ and } \frac{m(\Bin_n) - (1-c_{\bin}) n}{s_{\bin}\sqrt{n}} \xrightarrow{d} \mathcal{N}(0,1),$$
where $c_{\bin} = \frac{1+C_1}{2} \approx 0.542876$ and $s_{\bin} = \frac{K_2}{2} \approx 0.028581$.
\end{cor}

\section{Other eigenvalues}

For other eigenvalues, determining the mean and variance constants numerically is substantially more difficult as the recursive characterisation of the toll function is no longer available. It was remarked earlier that the mean constants can be determined from the convergent series
\[\mu_{\rec,\alpha} = \sum_{k=1}^{\infty} \frac{\Ex (n_{\alpha}(\Rec_k))}{k(k+1)} \text{ and } \mu_{\bin,\alpha} = \sum_{k=1}^{\infty} \frac{2 \Ex (n_{\alpha}(\Bin_k))}{(k+1)(k+2)}.\]
However, the convergence of these series is rather slow, and the numerators are not easy to compute either. To illustrate this, let us consider the example of the eigenvalue $1$ (or, by symmetry, the eigenvalue $-1$).

\medskip

By means of a computer program, we were able to compute the values $\Ex (n_{1}(\Rec_k))$ explicitly for $k \leq 30$. We have
\[\sum_{k=1}^{30} \frac{\Ex (n_1(\Rec_k))}{k(k+1)} \approx 0.048771.\]
However, the only a-priori estimate that we have to deal with the tail of the series is $|\Ex (n_1(\Rec_k))| \leq 1$, so
\[\Big| \sum_{k=31}^{\infty} \frac{\Ex (n_1(\Rec_k))}{k(k+1)} \Big| \leq \sum_{k=31}^{\infty} \frac{1}{k(k+1)} = \frac{1}{31},\]
which means that we only get the very weak bound $\mu_{\rec,1}  \in [0.016512,0.081029]$. By extra\-polation from the available numerical data, we obtain heuristically that $\mu_{\rec,1}$ is about $0.03$. It would certainly be desirable to have better methods to calculate the constants in our central limit theorems numerically.

\bibliographystyle{abbrv}
\bibliography{increasing_trees}

\end{document}